\documentclass{article}

\usepackage{amsmath}
\usepackage{amsthm}
\usepackage{alltt}
\usepackage{hyperref}
\usepackage{algorithm}
\usepackage{algpseudocode}
\usepackage{graphicx}

\numberwithin{equation}{section}

\newtheorem{theorem}{Theorem}[section]
\newtheorem{lemma}{Lemma}[section]
\newtheorem{definition}{Definition}[section]

\newcommand{\gaussian}[2]
{\genfrac{(}{)}{0pt}{}{#1}{#2}_{\textstyle q}}

\title{Computation of q-Binomial Coefficients\\ with the $P(n,m)$ Integer Partition Function}
\author{M.J. Kronenburg}
\date{}

\begin{document}

\maketitle

\begin{abstract}
Using $P(n,m)$, the number of integer partitions of $n$ into exactly $m$ parts,
which was the subject of an earlier paper,
$P(n,m,p)$, the number of integer partitions of $n$ into exactly $m$ parts with each part at most $p$,
can be computed in $O(n^2)$, and the q-binomial coefficient can be computed in $O(n^3)$.
Using the definition of the q-binomial coefficient, some properties of the q-binomial coefficient
and $P(n,m,p)$ are derived.
The q-multinomial coefficient can be computed as a product of q-binomial coefficients.
A formula for $Q(n,m,p)$, the number of integer partitions of $n$ into exactly $m$ distinct parts
with each part at most $p$, is given.
Some formulas for the number of integer partitions with each part between a minimum and a maximum are derived.
A computer algebra program is listed implementing these algorithms 
using the computer algebra program of the earlier paper.
\end{abstract}


\section{Definitions and Basic Identities}

Let the coefficient of a power series be defined as:
\begin{equation}
 [q^n] \sum_{k=0}^{\infty} a_k q^k = a_n
\end{equation}
Let $P(n)$ be the number of integer partitions of $n$, and let $P(n,m)$ be the
number of integer partitions of $n$ into exactly $m$ parts.
Let $P(n,m,p)$ be the number of integer partitions of $n$ into exactly $m$ parts with each part at most $p$,
and let $P^*(n,m,p)$ be the number of integer partitions of $n$ into at most $m$ parts with each part at most $p$,
which is the number of Ferrer diagrams that fit in a $m$ by $p$ rectangle:
\begin{equation}\label{pnmpsum}
 P^*(n,m,p) = \sum_{k=0}^m P(n,k,p)
\end{equation}
Let the following definition of the q-binomial coefficient,
also called the Gaussian polynomial, be given.
\begin{definition}
The q-binomial coefficient is defined by \cite{A84,AAR}:
\begin{equation}\label{gaussdef}
 \gaussian{m+p}{m} = \prod_{j=1}^m \frac{1-q^{p+j}}{1-q^j}
\end{equation}
\end{definition}
The q-binomial coefficient is the generating function of $P^*(n,m,p)$ \cite{A84}:
\begin{equation}\label{pnmpgen}
 P^*(n,m,p) = [q^n] \gaussian{m+p}{m}
\end{equation}
The q-binomial coefficient is a product of cyclotomic polynomials \cite{knuth}.

\section{Properties of q-Binomial Coefficients and $P(n,m,p)$}

Some identities of the q-binomial coefficient are proved from its definition,
and from these some properties of $P^*(n,m,p)$ and $P(n,m,p)$ are derived.
\begin{theorem}
\begin{equation}
 \gaussian{m+p}{m} = \gaussian{m+p-1}{m-1} + q^m \gaussian{m+p-1}{m}
\end{equation}
\end{theorem}
\begin{proof}
\begin{equation}
 \prod_{j=1}^m\frac{1-q^{p+j}}{1-q^j} 
 = \prod_{j=1}^{m-1}\frac{1-q^{p+j}}{1-q^j} + q^m \prod_{j=1}^m\frac{1-q^{p+j-1}}{1-q^j} 
\end{equation}
\begin{equation}
 \prod_{j=1}^m(1-q^{p+j}) = (1-q^m) \prod_{j=1}^{m-1}(1-q^{p+j}) + q^m \prod_{j=0}^{m-1}(1-q^{p+j})
\end{equation}
\begin{equation}
 1 = \frac{1-q^m}{1-q^{m+p}} + q^m \frac{1-q^p}{1-q^{m+p}}
\end{equation}
\begin{equation}
 1-q^{m+p} = 1-q^m+q^m(1-q^p)
\end{equation}
\end{proof}
\begin{theorem}
\begin{equation}
 P^*(n,m,p) = P^*(n,m-1,p) + P^*(n-m,m,p-1)
\end{equation}
\end{theorem}
\begin{proof}
Using the previous theorem:
\begin{equation}
\begin{split}
 P^*(n,m,p) & = [q^n]\gaussian{m+p}{m} = [q^n]\gaussian{m+p-1}{m-1} + [q^{n-m}]\gaussian{m+p-1}{m} \\
 & = P^*(n,m-1,p) + P^*(n-m,m,p-1) \\
\end{split}
\end{equation}
\end{proof}
From this theorem and identity (\ref{pnmpsum}) follows:
\begin{equation}\label{pnmpdif}
 P^*(n,m,p) - P^*(n,m-1,p) = P(n,m,p) = P^*(n-m,m,p-1)
\end{equation}
or equivalently:
\begin{equation}\label{pnmpdef}
 P^*(n,m,p) = P(n+m,m,p+1)
\end{equation}
From this theorem and this identity follows:
\begin{equation}
 P(n,m,p) = P(n-1,m-1,p) + P(n-m,m,p-1)
\end{equation}
\begin{theorem}
\begin{equation}
 \gaussian{m+p}{m} = \gaussian{m+p-1}{m} + q^p \gaussian{m+p-1}{m-1}
\end{equation}
\end{theorem}
\begin{proof}
\begin{equation}
 \prod_{j=1}^m\frac{1-q^{p+j}}{1-q^j} 
 = \prod_{j=1}^m\frac{1-q^{p+j-1}}{1-q^j} + q^p \prod_{j=1}^{m-1}\frac{1-q^{p+j}}{1-q^j} 
\end{equation}
\begin{equation}
 \prod_{j=1}^m(1-q^{p+j}) = \prod_{j=0}^{m-1}(1-q^{p+j}) + q^p(1-q^m) \prod_{j=1}^{m-1}(1-q^{p+j})
\end{equation}
\begin{equation}
 1 = \frac{1-q^p}{1-q^{m+p}} + q^p \frac{1-q^m}{1-q^{m+p}}
\end{equation}
\begin{equation}
 1-q^{m+p} = 1-q^p+q^p(1-q^m)
\end{equation}
\end{proof}
\begin{theorem}
\begin{equation}
 P^*(n,m,p) = P^*(n,m,p-1) + P^*(n-p,m-1,p)
\end{equation}
\end{theorem}
\begin{proof}
Using the previous theorem:
\begin{equation}
\begin{split}
 P^*(n,m,p) & = [q^n]\gaussian{m+p}{m} = [q^n]\gaussian{m+p-1}{m} + [q^{n-p}]\gaussian{m+p-1}{m-1} \\
 & = P^*(n,m,p-1) + P^*(n-p,m-1,p) \\
\end{split}
\end{equation}
\end{proof}
Using (\ref{pnmpdef}):
\begin{equation}
 P(n,m,p) = P(n,m,p-1) + P(n-p,m-1,p)
\end{equation}
The following theorem is a symmetry identity:
\begin{theorem}\label{gaussym}
\begin{equation}
 \gaussian{m+p}{m} = \gaussian{m+p}{p}
\end{equation}
\end{theorem}
\begin{proof}
\begin{equation}
 \prod_{j=1}^m \frac{1-q^{p+j}}{1-q^j} = \prod_{j=1}^p \frac{1-q^{m+j}}{1-q^j}
\end{equation}
Using cross multiplication:
\begin{equation}
 \prod_{j=1}^p(1-q^j)\prod_{j=1}^m(1-q^{p+j}) = \prod_{j=1}^m(1-q^j)\prod_{j=1}^p(1-q^{m+j})
 = \prod_{j=1}^{m+p}(1-q^j)
\end{equation}
\end{proof}
From this theorem follows:
\begin{equation}
 P^*(n,m,p) = P^*(n,p,m)
\end{equation}
and using (\ref{pnmpdef}):
\begin{equation}
 P(n,m,p) = P(n-m+p-1,p-1,m+1)
\end{equation}
Using (\ref{pnmpsum}) and (\ref{pnmpdif}):
\begin{equation}
 P^*(n,m,p) = \sum_{k=0}^m P^*(n-k,k,p-1)
\end{equation}
Combining this identity with (\ref{pnmpgen}) and using theorem \ref{gaussym}:
\begin{equation}
 \gaussian{m+p}{p} = \sum_{k=0}^m q^k \gaussian{p+k-1}{p-1}
\end{equation}
which is identity (3.3.9) in \cite{A84}.
Taking $m=n$ in (\ref{pnmpsum}) and (\ref{pnmpdef}) and conjugation of Ferrer diagrams:
\begin{equation}\label{pnnp}
 P(2n,n,p+1) = P(n+p,p)
\end{equation}
and taking $p=n$:
\begin{equation}
 P(n) = P(2n,n) = P(2n,n,n+1)
\end{equation}
The partitions of $P(n,m,p)-P(n,m,p-1)$ have at least one part equal to $p$,
and therefore by conjugation of Ferrer diagrams:
\begin{equation}
 P(n,m,p) - P(n,m,p-1) = P(n,p,m) - P(n,p,m-1)
\end{equation}
This identity can also be derived from the other identities.\\
Using (\ref{pnmpsum}) and (5.7) in \cite{AE04}:
\begin{equation}
 \sum_{m=0}^n P(n,m,p) = P^*(n,n,p) = [q^n] \frac{1}{\prod_{j=1}^p(1-q^j)}
\end{equation}
\begin{theorem}
\begin{equation}
 \sum_{m=0}^n (-1)^m P(n,m,p) = [q^n]  \frac{1}{\prod_{j=1}^p(1+q^j)}
\end{equation}
\end{theorem}
\begin{proof}
Using (\ref{pnmpdif}) and (\ref{pnmpgen}):
\begin{equation}
\begin{split}
 & \sum_{m=0}^n (-1)^m P(n,m,p) = \sum_{m=0}^n (-1)^m P^*(n-m,m,p-1) \\
 & = \sum_{m=0}^n (-1)^m [q^{n-m}] \gaussian{m+p-1}{m} 
 = [q^n] \sum_{m=0}^n (-1)^m q^m \gaussian{m+p-1}{m} \\
\end{split}
\end{equation}
Using the negative q-binomial theorem \cite{wiki1}:
\begin{equation}
 \sum_{m=0}^{\infty} \gaussian{m+p-1}{m} t^m = \frac{1}{\prod_{j=0}^{p-1} (1-q^jt)}
\end{equation}
Taking $t=-q$ and the sum up to $n$, because only the coefficient $[q^n]$ is needed, gives the theorem.
\end{proof}

\section{The q-Multinomial Coefficient}

Let $(m_i)_{i=1}^s$ be a sequence of $s$ nonnegative integers, and let $n$ be given by:
\begin{equation}
 n = \sum_{i=1}^s m_i
\end{equation}
The q-multinomial coefficient is a product of q-binomial coefficients \cite{CP,wiki2}:
\begin{equation}
 \gaussian{n}{m_1\cdots m_s} = \prod_{i=1}^s \gaussian{\sum_{j=1}^i m_j}{m_i}
 = \prod_{i=1}^s \gaussian{n-\sum_{j=1}^{i-1}m_j}{m_i}
\end{equation}

\section{Computation of $P(n,m,p)$ with $P(n,m)$}

Let the coefficient $a_k^{(m,p)}$ be defined as:
\begin{equation}
 a_k^{(m,p)} = [q^k] \prod_{j=1}^m (1-q^{p+j})
\end{equation}
These coefficients can be computed by multiplying out the product,
which up to $k=n-m$ is $O(m(n-m))=O(n^2)$.
Using (\ref{pnmpgen}) and (\ref{pnmpdif}):
\begin{equation}
\begin{split}
 P(n,m,p) & = P^*(n-m,m,p-1) = [q^{n-m}] \prod_{j=1}^m \frac{1-q^{p+j-1}}{1-q^j} 
 = [q^{n-m}] \sum_{k=0}^{n-m} a_k^{(m,p-1)} \frac{q^k}{\prod_{j=1}^m(1-q^j)} \\
 & = \sum_{k=0}^{n-m} a_k^{(m,p-1)} [q^{n-k}] \frac{q^m}{\prod_{j=1}^m(1-q^j)} 
  = \sum_{k=0}^{n-m} a_k^{(m,p-1)} P(n-k,m) \\
\end{split}
\end{equation}
The list of the $n-m+1$ values of $P(m,m)$ to $P(n,m)$ can be computed
using the algorithm in \cite{MK} which is also $O(n^2)$,
and therefore this algorithm computes $P(n,m,p)$ in $O(n^2)$.
For computing $P^*(n,m,p)$ (\ref{pnmpdef}) can be used.

\section{Computation of q-Binomial Coefficients}

From definition (\ref{gaussdef}) the q-binomial coefficients are:
\begin{equation}
\begin{split}
 [q^n] \prod_{j=1}^m \frac{1-q^{p+j}}{1-q^j}
 & = [q^n] \sum_{k=0}^n a_k^{(m,p)} \frac{q^k}{\prod_{j=1}^m(1-q^j)}
 = \sum_{k=0}^n a_k^{(m,p)} [q^{n+m-k}] \frac{q^m}{\prod_{j=1}^m(1-q^j)} \\
 & = \sum_{k=0}^n a_k^{(m,p)} P(n+m-k,m) \\
\end{split}
\end{equation}
Because $P^*(n,m,p)=0$ when $n>mp$ and (\ref{pnmpgen}), the coefficients $[q^n]$
are nonzero if and only if $0\leq n\leq mp$.
The product coefficients $a_k^{(m,p)}$ can therefore be computed in $O(m^2p)$,
and the list of $mp+1$ values of $P(m,m)$ to $P(mp+m,m)$ can also be computed
in $O(m^2p)$ \cite{MK}.
The sums are convolutions which can be done with \texttt{ListConvolve},
and therefore this algorithm computes the q-binomial coefficients in $O(m^2p)$.
Because of symmetry theorem \ref{gaussym}, $m$ and $p$ can be interchanged when $m>p$,
which makes the algorithm $O(\textrm{min}(m^2p,p^2m))$.
Using a change of variables:
\begin{equation}
 \gaussian{n}{m} = \gaussian{m+n-m}{m}
\end{equation}
The algorithm for computing this q-binomial coefficient is $O(\textrm{min}(m^2(n-m),(n-m)^2m))$.
From this follows that when $m$ or $n-m$ is constant, then the algorithm is $O(n)$, and when
$m=cn$ for some constant $c$, then the algorithm is $O(n^3)$. Because $P^*(n,m,p)=P^*(mp-n,m,p)$
only $P^*(n,m,p)$ for $0\leq n\leq \lceil mp/2\rceil$ needs to be computed,
which makes the algorithm about two times faster. 
For comparison of results with the computer algebra program below
an alternative algorithm using cyclotomic polynomials is given.
\begin{verbatim}
QBinomialAlternative[n_,m_]:=Block[{result={1},temp},
 Do[Which[Floor[n/k]-Floor[m/k]-Floor[(n-m)/k]==1,
  temp=CoefficientList[Cyclotomic[k,q],q];
  result=ListConvolve[result,temp,{1,-1},0]],{k,n}];
 result]
\end{verbatim}
Computations show that this alternative algorithm is $O(n^4)$.

\section{A Formula for $Q(n,m,p)$}

Let $Q(n,m,p)$ be the number of integer partitions of $n$ into exactly $m$
distinct parts with each part at most $p$.
\begin{theorem}\label{qnmp}
\begin{equation}
 Q(n,m,p) = P(n-m(m-1)/2,m,p-m+1)
\end{equation}
\end{theorem}
\begin{proof}
The proof is with Ferrer diagrams and the "staircase" argument.
Let a normal partition be a partition into $m$ parts,
and let a distinct partition be a partition into $m$ distinct parts.
Let the parts of a Ferrer diagram with $m$ parts be indexed from small to large by $s=1\cdots m$.
Each distinct partition of $n$ contains a "staircase" partition
with parts $s-1$ and a total size of $m(m-1)/2$, and subtracting this from such a
partition gives a normal partition of $n-m(m-1)/2$, and the largest part
is decreased by $m-1$.
Vice versa adding the "staircase" partition to a normal partition of $n$
gives a distinct partition of $n+m(m-1)/2$,
and the largest part is increased by $m-1$.
When the parts of the distinct partition are at most $p$,
then the parts of the corresponding normal partition
are at most $p-(m-1)$.
Because of this $1-1$ correspondence between the Ferrer diagrams of these two types
of partitions the identity is valid.
\end{proof}

\section{Partitions with Each Part Between $p_{\rm min}$ and $p_{\rm max}$}

Let $P^\#(n,p_{\min},p_{\max})$ be the number of partitions of $n$
with each part between $p_{\min}$ and $p_{\max}$, and let
$Q^\#(n,p_{\min},p_{\max})$ be the number of partitions of $n$
into distinct parts with each part between $p_{\min}$ and $p_{\max}$,
and let $P(n,m,p_{\min},p_{\max})$ be the number of partitions of $n$ into exactly $m$ parts
with each part between $p_{\min}$ and $p_{\max}$, 
and let $Q(n,m,p_{\min},p_{\max})$ be the number of partitions of $n$ into exactly $m$ distinct parts
with each part between $p_{\min}$ and $p_{\max}$.
These functions are related by:
\begin{equation}
 P^\#(n,p_{\min},p_{\max}) = \sum_{m=0}^{\lfloor n/p_{\min}\rfloor} P(n,m,p_{\min},p_{\max})
\end{equation}
\begin{equation}
 Q^\#(n,p_{\min},p_{\max}) = \sum_{m=0}^{\lfloor n/p_{\min}\rfloor} Q(n,m,p_{\min},p_{\max})
\end{equation}
Because the Ferrer diagrams of the partitions in $P(n,m,p_{\min},p_{\max})$ and $Q(n,m,p_{\min},p_{\max})$
all have a block of $m(p_{\min}-1)$ filled, the following relations are given:
\begin{equation}\label{pminmaxdef}
 P(n,m,p_{\min},p_{\max}) = P(n-m p_{\min}+m,m,p_{\max}-p_{\min}+1)
\end{equation}
\begin{equation}\label{qminmaxdef}
 Q(n,m,p_{\min},p_{\max}) = Q(n-m p_{\min}+m,m,p_{\max}-p_{\min}+1)
\end{equation}
\begin{theorem}\label{pnmpminpmax}
\begin{equation}
 P(n,m,p_{\min},p_{\max}) = \sum_{l=0}^{\min(m,p_{\min})} (-1)^l \sum_{k=l(l+1)/2}^{n-m+l} 
 P(k-l(l-1)/2,l,p_{\min}-l)P(n-k,m-l,p_{\max})
\end{equation}
\end{theorem}
\begin{proof}
From (5.11) in \cite{AE04}:
\begin{equation}
 \sum_{n=0}^{\infty}\sum_{m=0}^{\infty} P(n,m,p_{\min},p_{\max}) q^nz^m
 = \frac{1}{\prod_{j=p_{\min}}^{p_{\max}}(1-zq^j)}
 = \frac{\prod_{j=1}^{p_{\min}-1}(1-zq^j)}{\prod_{j=1}^{p_{\max}}(1-zq^j)}
\end{equation}
From (5.9) and (5.11) in \cite{AE04}:
\begin{equation}
 \prod_{j=1}^{p_{\min}-1}(1-zq^j) = \sum_{n=0}^{\infty}\sum_{m=0}^{\infty} Q(n,m,p_{\min}-1) (-1)^m z^mq^n
\end{equation}
\begin{equation}
 \frac{1}{\prod_{j=1}^{p_{\max}}(1-zq^j)} = \sum_{n=0}^{\infty}\sum_{m=0}^{\infty} P(n,m,p_{\max}) z^mq^n
\end{equation}
\begin{equation}
\begin{split}
  & \sum_{n=0}^{\infty}\sum_{m=0}^{\infty} P(n,m,p_{\min},p_{\max}) q^nz^m \\
 & = ( \sum_{n=0}^{\infty}\sum_{m=0}^{\infty} Q(n,m,p_{\min}-1) (-1)^m q^nz^m )
   ( \sum_{n=0}^{\infty}\sum_{m=0}^{\infty} P(n,m,p_{\max}) q^nz^m ) \\
 & = \sum_{n_1=0}^{\infty}\sum_{n_2=0}^{\infty}\sum_{m_1=0}^{\infty}\sum_{m_2=0}^{\infty}
  Q(n_1,m_1,p_{\min}-1) P(n_2,m_2,p_{\max}) (-1)^{m_1} q^{n_1+n_2} z^{m_1+m_2} \\
\end{split}
\end{equation}
The coefficients on both sides must be equal, so $n_1+n_2=n$ and $m_1+m_2=m$,
which is equivalent to $n_2=n-n_1$ and $m_2=m-m_1$:
\begin{equation}\label{pnmpminpmaxeq}
 P(n,m,p_{\min},p_{\max}) = \sum_{l=0}^m (-1)^l \sum_{k=l}^{n-m+l} Q(k,l,p_{\min}-1) P(n-k,m-l,p_{\max})
\end{equation}
Application of theorem \ref{qnmp} to this identity gives this theorem.
\end{proof}
The following is a special case of $P(n,m,p_{\min},p_{\max})$:
\begin{equation}
 P(n,m,p,p) = 
\begin{cases}
 1 & \textrm{\rm if $mp=n$} \\
 0 & \textrm{\rm otherwise} \\
\end{cases}
\end{equation}
\begin{theorem}\label{qnmpminpmax}
Let $P(n,m,p_{\min},p_{\max})$ be the formula of the previous theorem:
\begin{equation}
 Q(n,m,p_{\min},p_{\max}) = (-1)^m P(n,m,p_{\max}+1,p_{\min}-1)
\end{equation}
\end{theorem}
\begin{proof}
\begin{equation}
 \sum_{n=0}^{\infty}\sum_{m=0}^{\infty} Q(n,m,p_{\min},p_{\max}) q^nz^m
 = \prod_{j=p_{\min}}^{p_{\max}} (1+zq^j)
 = \frac{\prod_{j=1}^{p_{\max}}(1+zq^j)}{\prod_{j=1}^{p_{\min}-1}(1+zq^j)}
\end{equation}
\begin{equation}
 \prod_{j=1}^{p_{\max}}(1+zq^j) = \sum_{n=0}^{\infty}\sum_{m=0}^{\infty} Q(n,m,p_{\max}) z^mq^n
\end{equation}
\begin{equation}
 \frac{1}{\prod_{j=1}^{p_{\min}-1}(1+zq^j)} = \sum_{n=0}^{\infty}\sum_{m=0}^{\infty} P(n,m,p_{\min}-1) (-1)^m z^mq^n
\end{equation}
\begin{equation}
\begin{split}
  & \sum_{n=0}^{\infty}\sum_{m=0}^{\infty} Q(n,m,p_{\min},p_{\max}) q^nz^m \\
 & = ( \sum_{n=0}^{\infty}\sum_{m=0}^{\infty} Q(n,m,p_{\max}) q^nz^m )
   ( \sum_{n=0}^{\infty}\sum_{m=0}^{\infty} P(n,m,p_{\min}-1) (-1)^m q^nz^m ) \\
 & = \sum_{n_1=0}^{\infty}\sum_{n_2=0}^{\infty}\sum_{m_1=0}^{\infty}\sum_{m_2=0}^{\infty}
  Q(n_1,m_1,p_{\max}) P(n_2,m_2,p_{\min}-1) (-1)^{m_2} q^{n_1+n_2} z^{m_1+m_2} \\
\end{split}
\end{equation}
The coefficients on both sides must be equal, so $n_1+n_2=n$ and $m_1+m_2=m$,
which is equivalent to $n_2=n-n_1$ and $m_2=m-m_1$:
\begin{equation}\label{qnmpminpmaxeq}
 Q(n,m,p_{\min},p_{\max}) = (-1)^m \sum_{l=0}^m (-1)^l \sum_{k=l}^{n-m+l} Q(k,l,p_{\max}) P(n-k,m-l,p_{\min}-1)
\end{equation}
Comparing this with (\ref{pnmpminpmaxeq}) in theorem \ref{pnmpminpmax} gives this theorem.
\end{proof}
From the generating functions in \cite{AE04} follows:
\begin{equation}
\begin{split}
 & P^\#(n,p_{\min},p_{\max}) = [q^n] \frac{1}{\prod_{j=p_{\min}}^{p_{\max}}(1-q^j)}
 = [q^n] \frac{\prod_{j=1}^{p_{\min}-1}(1-q^j)}{\prod_{j=1}^{p_{\max}}(1-q^j)} \\
 & = [q^n] \sum_{k=0}^n q^k a_k^{(p_{\min}-1,0)} \frac{1}{\prod_{j=1}^{p_{\max}}(1-q^j)} 
 = \sum_{k=0}^n a_k^{(p_{\min}-1,0)} [q^{p_{\max}+n-k}] \frac{q^{p_{\max}}}{\prod_{j=1}^{p_{\max}}(1-q^j)} \\
 & = \sum_{k=0}^n a_k^{(p_{\min}-1,0)} P(p_{\max}+n-k,p_{\max}) \\
\end{split}
\end{equation}
\begin{equation}
 Q^\#(n,p_{\min},p_{\max}) = [q^n] \prod_{j=p_{\rm min}}^{p_{\max}} (1+q^j)
  = [q^n] \prod_{j=1}^{p_{\max}-p_{\min}+1} (1+q^{p_{\min}-1+j})
\end{equation}
The following are special cases of $P^\#(n,p_{\min},p_{\max})$:
\begin{equation}
 P^\#(n,1,m) = P(n+m,m)
\end{equation}
\begin{equation}
 P^\#(n,p,p) = 
\begin{cases}
 1 & \textrm{\rm if $p$ divides $n$} \\
 0 & \textrm{\rm otherwise} \\
\end{cases}
\end{equation}
The following is lemma (5) in \cite{A83}:
\begin{lemma}\label{mylemma}
\begin{equation}
 \sum_{k=1}^m q^k \prod_{j=1}^{k-1} (1-q^j) = 1 - \prod_{j=1}^m (1-q^j)
\end{equation}
\end{lemma}
\begin{proof}
The lemma is true for $m=0$, and using induction on $m$,
when it is true for $m$, then for $m+1$:
\begin{equation}
\begin{split}
 & \sum_{k=1}^{m+1} q^k \prod_{j=1}^{k-1} (1-q^j) 
 = q^{m+1} \prod_{j=1}^m (1-q^j) + \sum_{k=1}^m q^k \prod_{j=1}^{k-1} (1-q^j) \\
 & = q^{m+1} \prod_{j=1}^m (1-q^j) + 1 - \prod_{j=1}^m (1-q^j) 
 = 1 - (1-q^{m+1}) \prod_{j=1}^m (1-q^j) = 1 - \prod_{j=1}^{m+1} (1-q^j) \\
\end{split}
\end{equation}
\end{proof}
Dividing this lemma by $\prod_{j=1}^m(1-q^j)$ and taking the coefficient $[q^n]$:
\begin{equation}
 [q^n] \sum_{k=1}^m q^k \frac{1}{\prod_{j=k}^m(1-q^j)} = \sum_{k=1}^m [q^{n-k}] \frac{1}{\prod_{j=k}^m(1-q^j)}
 = [q^n] \frac{1}{\prod_{j=1}^m(1-q^j)} - [q^n] 1
\end{equation}
and using $P(0,p,p)=1$ gives the following identity:\\
For integer $m\leq n$:
\begin{equation}
 \sum_{k=1}^{\min(m,\lfloor n/2\rfloor)} P^\#(n-k,k,m) = P^\#(n,1,m) - \delta_{n,m}
\end{equation}
Taking $m=n$ and using $P^\#(n,1,n)=P(n)$:
\begin{equation}
 \sum_{k=1}^{\lfloor n/2\rfloor} P^\#(n-k,k,n) = P(n) - 1
\end{equation}
The following was proved as lemma 1.1 in \cite{MK}:
\begin{equation}
 \sum_{k=0}^m q^k \prod_{j=k+1}^m (1-q^j) = 1
\end{equation}
Dividing this lemma by $\prod_{j=1}^m(1-q^j)$ and taking the coefficient $[q^n]$:\\
For integer $m\leq n$:
\begin{equation}
 \sum_{k=1}^m P^\#(n-k,1,k) = P^\#(n,1,m)
\end{equation}
Taking $m=n$ and using $P^\#(n,1,n)=P(n)$:
\begin{equation}
 \sum_{k=1}^n P^\#(n-k,1,k) = P(n)
\end{equation}

\section{The q-Binomial Coefficient for Negative Arguments}

From another paper \cite{MK2} the following was proved:\\
For integer $n\geq 0$ and integer $k$:
\begin{equation}
 \gaussian{n}{k} = 0 \textrm{~if~} k<0 \textrm{~or~} k>n
\end{equation}
and from that paper theorem 2.4 gives:\\
For negative integer $n$ and integer $k$:
\begin{equation}
 \gaussian{n}{k} =
 \begin{cases}
   \displaystyle (-1)^k q^{nk-k(k-1)/2} \gaussian{-n+k-1}{k} & \text{if $k\geq 0$} \\
   \displaystyle (-1)^{n-k} q^{(n-k)(n+k+1)/2} \gaussian{-k-1}{n-k} & \text{if $k\leq n$} \\
   0 & \text{otherwise} \\
 \end{cases}
\end{equation}

\section{Computer Algebra Program}

The following Mathematica\textsuperscript{\textregistered}
functions are listed in the computer algebra program below.\\
\texttt{PartitionsPList[n,pmin,pmax]}\\
Gives a list of the $n$ numbers $P(1,p_{\min},p_{\max})..P(n,p_{\min},p_{\max})$, where
$P(n,p_{\min},p_{\max})$ is the number of partitions of $n$ with each part between $p_{\min}$ and $p_{\max}$.
This algorithm is $O(n^2)$.\\
\texttt{PartitionsQList[n,pmin,pmax]}\\
Gives a list of the $n$ numbers $Q(1,p_{\min},p_{\max})..Q(n,p_{\min},p_{\max})$, where
$Q(n,p_{\min},p_{\max})$ is the number of partitions of $n$ into distinct parts
with each part between $p_{\min}$ and $p_{\max}$.
This algorithm is $O(n^2)$.\\
\texttt{PartitionsInPartsP[n,m,p]}\\
Gives the number of partitions of $n$ into exactly $m$ parts
with each part at most $p$. This algorithm is $O(n^2)$.\\
\texttt{PartitionsInPartsQ[n,m,p]}\\
Gives the number of partitions of $n$ into exactly $m$ distinct parts
with each part at most $p$,
using the formula $Q(n,m,p)=P(n-m(m-1)/2,m,p-m+1)$. This algorithm is $O(n^2)$.\\
\texttt{PartitionsInPartsPList[n,m,p]}\\
Gives a list of $n-m+1$ numbers of $P(m,m,p)$..$P(n,m,p)$.
This algorithm is $O(n^2)$.\\
\texttt{PartitionsInPartsQList[n,m,p]}\\
Gives a list of the $n-m(m+1)/2+1$ numbers $Q(m(m+1)/2,m,p)..Q(n,m,p)$,
using the formula $Q(n,m,p)=P(n-m(m-1)/2,m,p-m+1)$.
This algorithm is $O(n^2)$.\\
\texttt{QBinomialCoefficients[n,m,q]}\\
Gives the q-binomial coefficient $\binom{n}{m}_q$ as a polynomial in $q$,
where $n$ and $m$ are integers and $q$ is a symbol.
This algorithm is $O(n^3)$.\\
\texttt{QMultinomialCoefficients[mlist,q]}\\
Gives the q-multinomial coefficient $\binom{n}{m_1\cdots m_s}_q$ 
as a polynomial in $q$, where $s$ is the length of the list \texttt{mlist}
containing the integers $m_1\cdots m_s$, and where $n=\sum_{i=1}^sm_i$,
and where $q$ is a symbol.\\

Below is the listing of a Mathematica\textsuperscript{\textregistered} program
that can be copied into a notebook, using the package
taken from at least version 3 of the earlier paper \cite{MK}. The notebook must be saved
in the directory of the package file.

\begin{verbatim}
SetDirectory[NotebookDirectory[]];
<< "PartitionsInParts.m"
partprod[n_,m_,p_,s_]:=Block[{prod=ConstantArray[0,n+1]},prod[[1]]=1;
 Do[prod[[Range[p+k+1,n+1]]]+=s prod[[Range[1,n-p-k+1]]],{k,Min[m,n-p]}];
 prod]
PartitionsPList[n_Integer?Positive,pmin_Integer?Positive, 
 pmax_Integer?Positive]:=If[pmax<pmin,{}, 
 Block[{prods,parts},prods=partprod[n,pmin-1,0,-1];
  parts=PartitionsInPartsPList[n+pmax,pmax];
  ListConvolve[prods,parts,{1,1},0][[Range[2,n+1]]]]]
PartitionsQList[n_Integer?Positive,pmin_Integer?Positive, 
 pmax_Integer?Positive]:=If[pmax<pmin,{}, 
  partprod[n,pmax-pmin+1,pmin-1,1][[Range[2,n+1]]]]
PartitionsInPartsP[n_Integer?NonNegative,m_Integer?NonNegative, 
 p_Integer?NonNegative]:=If[n<m,0,
 Block[{prods,parts,result},prods=partprod[n-m,m,p-1,-1];
  parts=PartitionsInPartsPList[n,m];result=0;
  Do[result+=prods[[k+1]]parts[[n-m-k+1]],{k,0,n-m}];result]]
PartitionsInPartsQ[n_Integer?NonNegative,m_Integer?NonNegative, 
 p_Integer?NonNegative]:=If[n-m(m-1)/2<m||p<m,0,
  PartitionsInPartsP[n-m(m-1)/2,m,p-m+1]]
PartitionsInPartsPList[n_Integer?NonNegative,m_Integer?NonNegative, 
 p_Integer?NonNegative]:=If[n<m,{},
 Block[{prods,parts},prods=partprod[n-m,m,p-1,-1];
  parts=PartitionsInPartsPList[n,m];ListConvolve[prods,parts,{1,1},0]]]
PartitionsInPartsQList[n_Integer?NonNegative,m_Integer?NonNegative, 
 p_Integer?NonNegative]:=If[n-m(m-1)/2<m||p<m,{}, 
  PartitionsInPartsPList[n-m(m-1)/2,m,p-m+1]]
QBinomialCompute[N_Integer,M_Integer,q_Symbol,doq_?BooleanQ]:= 
 Block[{m=M,p=N-M,result,prods,parts,ceil},Which[m>p,m=p;p=M];
  ceil=Ceiling[(m p+1)/2];prods=partprod[ceil-1,m,p,-1];
  parts=PartitionsInPartsPList[ceil+m-1,m];
  result=PadRight[ListConvolve[prods,parts,{1,1},0],m p+1];
  result[[Range[ceil+1,m p+1]]]=result[[Range[m p-ceil+1,1,-1]]];
  If[doq,q^Range[0,Length[result]-1].result,result]]
QBinomialCoefficients[n_Integer,k_Integer,q_Symbol]:=
 If[(n>=0&&(k<0||k>n))||(n<0&&n<k<0),0,If[n>=0,QBinomialCompute[n,k,q,True], 
 If[k>=0,(-1)^k q^(n k-k(k-1)/2)QBinomialCompute[-n+k-1,k,q,True],
 (-1)^(n-k)q^((n-k)(n+k+1)/2)QBinomialCompute[-k-1,n-k,q,True]]]]
MListQ[alist_List]:=(alist!={}&&VectorQ[alist,IntegerQ])
QMultinomialCoefficients[mlist_List?MListQ,q_Symbol]:=
 If[!VectorQ[mlist,NonNegative],0,
 Block[{length=Length[mlist],bprod={1},msum=mlist[[1]],qbin},
  Do[msum+=mlist[[k]];qbin=QBinomialCompute[msum,mlist[[k]],q,False];
   bprod=ListConvolve[bprod,qbin,{1,-1},0],{k,2,length}];
  q^Range[0,Length[bprod]-1].bprod]]
\end{verbatim}

\pdfbookmark[0]{References}{}

\end{document}